\documentclass[12pt,a4paper]{article}
\usepackage{amsthm}
\usepackage{amssymb}
\usepackage{amsrefs}
\usepackage{latexsym}
\usepackage{xspace}
\usepackage{bbm}
\usepackage{amssymb}
\usepackage{amsmath}
\usepackage{graphicx}
\usepackage{epic}
\usepackage{eepic}
\usepackage{multirow}
\usepackage{gastex}
\usepackage{times}
\usepackage{thmtools, thm-restate}

\declaretheorem{theorem}
\declaretheorem[sibling=theorem]{lemma}
\declaretheorem[sibling=theorem]{proposition}
\declaretheorem[sibling=theorem]{definition}

\newcommand{\TODO}[1]{}

\title{On the Rigidity of Sparse Random Graphs}
\author{
Nati Linial
\thanks{Department of Computer Science, Hebrew University, Jerusalem 9190401. e-mail: nati@cs.huji.ac.il. Supported by a grant from the Israel Science Foundation.}
\and
Jonathan Mosheiff
\thanks{Department of Computer Science, Hebrew University, Jerusalem 9190401. e-mail: yonatanm@cs.huji.ac.il. Supported by the Adams Fellowship Program of the Israel Academy of Sciences and Humanities.}}
\date{}

\begin{document}

\maketitle

\begin{abstract}
A graph with a trivial automorphism group is said to be {\em rigid}. Wright proved \cite{Wri71} that for $\frac{\log n}{n}+\omega(\frac 1n)\leq p\leq \frac 12$ a random graph $G\in G(n,p)$ is rigid whp. It is not hard to see that this lower bound is sharp and for $p<\frac{(1-\epsilon)\log n}{n}$ with positive probability $\text{aut}(G)$ is nontrivial. We show that in the sparser case $\omega(\frac 1 n)\leq p\leq \frac{\log n}{n}+\omega(\frac 1n)$, it holds whp that $G$'s $2$-core is rigid. We conclude that for all $p$, a graph in $G(n,p)$ is reconstrutible whp. In addition this yields for $\omega(\frac 1n)\leq p\leq \frac 12$ a canonical labeling algorithm that almost surely runs in polynomial time with $o(1)$ error rate. This extends the range for which such an algorithm is currently known~\cite{CP08}.
\end{abstract}

\section{Introduction}

It is a truth universally acknowledged, that random objects are asymmetric. It was shown by Wright \cite{Wri71} that for $\frac 12\ge p\geq \frac {\log n}{n} + \omega(\frac 1n)$ a random $G(n,p)$ graph has, whp, a trivial automorphism group. He actually worked with the $G(n,M)$ model, but the reduction to $G(n,p)$ is well-known and follows easily from the Chernoff bound. Also, a graph and its complement clearly have the same automorphism group, so we can restrict ourselves to the range $\frac 12\ge p$. Wright's bound is tight, since a graph $G$ of slightly smaller density is likely to have isolated vertices, which can be swapped by a $G$-automorphism. This paper concerns the range of smaller $p$ by showing that for $\omega(\frac 1n)\leq p\leq n^{-\frac 12-\epsilon}$ whp all of $G$'s automorphisms are essentially trivial. Here is our main result:

\begin{restatable}{theorem}{mainTheorem}\label{thm:mainTheorem}
Let $G=(V,E)$ be a $G(n,p)$ graph with $\omega(\frac 1n)\le p\le n^{-\frac 12-\epsilon}$. Then whp its $2$-core has a trivial automorphism group.
\end{restatable}

This shows that for this range of $p$, whp $\text{aut}(G)$ is generated by:
\begin{itemize}
	\item Automorphisms of rooted trees that are attached to the $2$-core.
	\item Automrophisms of the tree components and swaps of such components.
\end{itemize}

The most interesting range of this statement is $p\le \frac{\log n+(1+\epsilon)\log\log n}{n}$. For larger $p$ the $2$-core is the whole graph, in which range ours is just a new proof for the rigidity of sufficiently dense random graphs. 

{\bf General strategy of the proof:} We denote the vertex set of $G$'s $2$-core by $R(G)$. It is easy to see that $\text{aut}(G)$ fixes $R(G)$ setwise and our proof shows first that $\text{aut}(G)$ actually fixes $R(G)$ pointwise. In order to prove the theorem in full we show that this rigidity does not result from boundary effects of vertices near $V\setminus R(G)$. The neighbor set of $v\in V$ and its degree are denoted by $N(v)$ and $d(v)$. If $x_1,\ldots,x_k$ are the neighbors of $v$, we denote by $\nabla(v)$ the {\em multiset} $\{d(x_i)\}_1^k$. Clearly $\nabla$ is preserved by automorphisms. We fix some $k\leq \log n$ and consider two directed rooted cycles $v_1,\ldots v_k$ and $u_1\ldots u_k$ in $G$. We show that whp every two such cycles have many {\em incompatible }pairs $(v_i,u_i)$ for which $\nabla(v_i)\neq\nabla(u_i)$. This already implies that $R(G)$ is fixed pointwise. In the full proof of the theorem we find, for every two such cycles, an incompatible pair $(v_i,u_i)$, where both $v_i$ and $u_i$ are at distance $\ge 3$ from $V\setminus R(G)$. Such a pair is not only incompatible in $G$, but also in $R(G)$, proving the theorem.

It turns out that Theorem \ref{thm:mainTheorem} yields some interesting insights on the well-known graph reconstruction conjecture which we now recall. Let $G$ be an $n$-vertex graph. When we delete a vertex of $G$ we obtain an $(n-1)$-vertex graph. By doing this separately for each vertex in $G$ we obtain the $n$ graphs that make up $G$'s {\em deck}. The {\em graph reconstruction conjecture}~(\cite{Kel57},~\cite{Ula60}) posits that every two graphs of $3$ or more vertices that have identical decks must be isomorphic. A graph $G$ is said to be {\em reconstructible} if every graph with the same deck is isomorphic to $G$. Bollob\'as proved \cite{Bol90} that whp $G(n,p)$ graphs are reconstructible for all $\frac{(5/2+\epsilon)\log n}{n}\leq p\leq 1-\frac{(5/2+\epsilon)\log n}{n}$. We show that this is in fact true for every $0\leq p\leq 1$. One reason why this extension of range is of interest has to do with the {\em edge reconstruction conjecture}~\cite{Har64} which states that every graph can be reconstructed from its deck of edge-deleted subgraphs. This leads to the notion of {\em edge-reconstructible} graphs. We recall two facts from this theory: (i) Every reconstructible graph with no isolated vertices is edge-reconstructible~(e.g., \cite{Bon91}) (ii) Every $n$-vertex graph with at least $\log_2 (n!) + 1=n\log_2 n+O(n)$ edges is edge reconstructible. Our result applies to the range $|E|\le O(n\log n)$ where the edge reconstruction problem is still open.

We turn to discuss the {\em canonical labeling problem} \cite{BES80}. Let $\cal L$ be a class of graphs. A canonical labeling of $G\in \cal L$ assigns distinct labels to the vertices of $G$, where the labeling is uniquely determined by $G$'s isomorphism class. In the probabilistic version of this problem, $\cal L$ is a probability space of graphs and we seek to efficiently find a canonical labeling for almost all graphs in $\cal L$. Such a canonical labeling algorithm clearly solves in particular the {\em random graph isomorphism problem} for $\cal L$. Specifically we ask for which values of $p$  there is a polynomial time canonical labeling in $G(n,p)$. By considering the complementary graph it suffices to consider the range $p\leq 1/2$. Such an algorithm is known \cite{CP08} for $p\in[\Theta(\frac{\ln n}{n}),1/2]$. Our proof of Theorem \ref{thm:mainTheorem} yields a polynomial time algorithm for $\omega(\frac 1n)\leq p\leq n^{-(0.5+\epsilon)}$, whence a polynomial time solution exists for $p\in [\omega(\frac 1n),\frac 12]$.

\section{Technical Preliminaries}

\paragraph{Graph theory:} Graphs are denoted $G=(V,E)$ and usually $n:=|V|$. The neighbor set of $u\in V$ is denoted by $N(u)$.
For $U\subseteq V$, we denote $N(U):=(\bigcup_{u\in U} N(u))\setminus U$ and $\tilde{N}(U) :=\bigcup_{u\in U} \tilde{N}(u)$.

The set of cross edges between two subsets $U,W\subseteq V$ is denoted $E(U,W):=\{uv\in E\mid u\in U, v\in V\}$, and $d(U,W)=|E(U,W)|$ (to wit: even if $U\cap W\neq\emptyset$, we consider every relevant edge exactly once). For a singleton $U=\{u\}$, we use the shorthand $d(u,W)=d(U,W)$. Also, $E(U)=E(U,U)$.

For $U\subseteq V$ we denote $\sigma(U) := \{v\in V\setminus{U}\mid d(v,U)=1\}$, the set of those vertices not in $U$ that have exactly one neighbor in $U$.

We denote by $G_U$ the subgraph of $G$ induced by $U\subseteq V$. 

Let $\nabla(u)$ denote the \emph{multiset of integers} $\{d(v,V\setminus\tilde{N}(u)) \mid v\in N(u)\}$.

We denote the vertex set of $G$'s $2$-core by $R(G)$.

\paragraph{Asymptotics:} A property of $G(n,p)$ graphs is said to hold {\em whp} ({\em with high probability}) if its probability tends to $1$ as $n\to \infty$.

\paragraph{Probability:} For a discrete random variable $X$, let
\[
\Pi(X)=\sup_{x\in\text{range}(X)}(\Pr(X=x)).
\]

If $X$ is multinomial with parameters $(m,(p_{1},\ldots,p_{k}))$, we denote $\Pi(X)$ by $\Pi(m,(p_1 ,\ldots ,p_k ))$. The following lemmas provide a description of $\Pi(m,(p_1 ,\ldots ,p_k ))$.

\begin{lemma}
\label{lem:MultinomialMaximizer}Let $X$ be a multinomial random variable with parameters $(m,(p_{1},\ldots,p_{k}))$ and suppose that $\Pi(X)=\Pr(X=(a_{1},\ldots a_{k}))$. Then, $a_{t}>mp_{t}-1$ for every $t$, or, in other words $a_{t}\ge \lfloor mp_{t}\rfloor$.\end{lemma}
\begin{proof}
Without loss of generality, assume by contradiction that $a_{1}\leq m\cdot p_{1}-1$.
Then, since $\sum_{i}a_{i}=m=\sum_{i}m\cdot p_{i}$, there exists
some index $s$, say $s=2$ such that $a_{s}>m\cdot p_{s}$.
\[
\frac{\Pr(X=(a_{1}+1,a_{2}-1,a_{3},\ldots,a_{k}))}{\Pr(X=(a_{1},\ldots,a_{k}))}=\frac{a_{2}}{a_{1}+1}\cdot\frac{p_{1}}{p_{2}}>\frac{m\cdot p_{2}}{m\cdot p_{1}}\cdot\frac{p_{1}}{p_{2}}=1
\]
contrary to the assumed maximality of $\Pr(X=(a_{1},\ldots,a_{k}))$.\end{proof}

\begin{lemma}
\label{lem:MultinomialBound}For an integer $m$, a constant $c>0$ and a probability vector ${\bf p}=(p_1 ,\ldots ,p_k)$, such that $p_i \leq \frac{c}{\sqrt m}$ for each $i$, it holds that
\[
\Pi(m,{\bf p})\leq m^{-\Omega\left(\frac{\sqrt{m}}{c}\right)}.
\]
\end{lemma}
\begin{proof}
We first show how to reduce the proof to the case where $p_i \geq \frac {c}{3\sqrt m}$ for each $i$. Assume that the lemma holds in this case. For a real vector ${\bf u}$ and two coordinate indices $i\neq j$, let ${\bf u}_{i,j}$ be the vector obtained by eliminating the coordinates $u_i, u_j$ and introducing a new coordinate of $u_i+u_j$. Let $X$ and $X_{i,j}$ be multinomial random variables with parameters $(m,{\bf p})$, $(m,{\bf p}_{i,j})$, respectively. Note that for every ${\bf a} \in \text{range}(X)$ there holds
\[
\Pr(X={\bf a})\leq\Pr(X_{i,j}={\bf a}_{i,j}).
\]
Thus, $\Pi(m,{\bf p})\leq \Pi({\bf p}_{i,j})$.

We generate a sequence of probability vectors that start from ${\bf p}$ and proceed as follows. At each step we replace, as described, the two smallest coordinates in the present probability vector by one coordinate that is their sum. We continue with this process until the first time at which this vector ${\bf q}$ has at most one coordinate that is smaller than $\frac{c}{2 \sqrt m}$. If the smallest coordinate in ${\bf q}$ is $\ge\frac{c}{3 \sqrt m}$, then, since each of the above steps can only increase $\Pi$, the reduction is complete.
Otherwise, ${\bf q}$ has exactly one coordinate, say $q_1$ that is $< \frac{c}{2 \sqrt m}$. But then in ${\bf q}_{1,2}$ all coordinates vary between $\frac{c}{2 \sqrt m}$ and $\frac{3c}{2 \sqrt m}$. The reduction is again complete.

We now turn to proving the lemma for the case where $\frac{c}{\sqrt{m}}\ge p_i \geq \frac{c}{3\sqrt m}$ for $i=1,\ldots,k$. Clearly $k\geq\frac{\sqrt{m}}{c}$.
Set $\mu_{i}=p_{i}\cdot m$ and suppose that $\Pi(X)=\Pr(X=(a_{1},\ldots,a_{k})).$ By Lemma \ref{lem:MultinomialMaximizer}, $\frac{a_{i}}{\mu_{i}}\ge1-\frac{1}{\mu_{i}}\geq1-\frac{3}{c\sqrt{m}}$ for all $i$.
Now
\[
\Pi(X)=\Pr(X=(a_{1},\ldots,a_{k}))={m \choose a_{1},\ldots,a_{k}}\cdot\prod_{i}p_{i}^{a_{i}}.
\]
By Stirling's bound, $\frac{n!}{\left(\frac{n}{e}\right)^{n}\sqrt{2\pi n}}=1+O(\frac{1}{n})$.
Thus,
\begin{eqnarray*}
\Pi(X) & \leq & O\left(\frac{\left(\frac{m}{e}\right)^{m}\sqrt{2\pi m}}{\prod_{i}\left(\frac{a_{i}}{e\cdot p_{i}}\right)^{a_{i}}\cdot\sqrt{2\pi a_i}}\right),
\end{eqnarray*}
which can be stated as
\begin{eqnarray*}
\Pi(X) & \leq & O\left(\frac{\left(\frac{m}{e}\right)^{m}\sqrt{2\pi m}}{\prod_{i}\left(\frac{\mu_{i}}{e\cdot p_{i}}\right)^{a_{i}}\cdot\sqrt{2\pi a_i}}\cdot\prod_{i}\left(\frac{\mu_{i}}{a_{i}}\right)^{a_{i}}\right)
=  O\left(\frac{\sqrt{2\pi m}}{\prod_{i}\sqrt{2\pi a_i}}\cdot\prod_{i}\left(\frac{\mu_{i}}{a_{i}}\right)^{a_{i}}\right).
\end{eqnarray*}
But
\[
\prod_{i}\left(\frac{\mu_{i}}{a_{i}}\right)^{a_{i}}\leq\left(1+\frac{4}{c\sqrt{m}}\right)^{m}\leq e^{\frac{4\sqrt{m}}{c}}
\]
and
\[
\frac{1}{\prod_{i}\sqrt{a_{i}}}\leq\left(\frac{c}{3}\sqrt{m}-1\right)^{-\frac{k}{2}}\leq \left(\frac{c}{3}\sqrt{m}-1\right)^{-\frac{\sqrt{m}}{2c}}\leq O\left(\left(\frac{c^{2}m}{9}\right)^{-\frac{\sqrt{m}}{4c}}\right).
\]
Therefore,
\[
\Pi(X)\leq O\left( (2\pi)^{-k/2}\sqrt m \cdot e^{\frac{4\sqrt m}{c}} \cdot \left(\frac{c^2 m}{9}\right)^{-\frac{4\sqrt m}{c}}\right)\le m^{-\Omega\left(\frac{\sqrt{m}}{c}\right)}.
\]

 \end{proof}

\begin{lemma}\label{Lem:MultinomialWithBinomialProbsBound}
Let $k$ be an integer, $\frac{1}{2}\ge p>0$ and let $p_{i}={k \choose i}p^{i}q^{k-i}~~~(i=0,1\ldots,k)$, where $q=1-p$.
Then, for every $m\le O(kp)$ there holds
\[
\Pi(m,(p_{0},\ldots,p_{k}))\leq m^{-\Omega(\sqrt{m})}.
\]
\end{lemma}

\begin{proof}
It is well known that
$$
\Pi(k,(p,q)) \leq O\left(\left( \frac{1}{pk-1}\right)^{1/2}\right)\leq O\left(\frac 1{\sqrt m}\right).
$$
Therefore, by Lemma \ref{lem:MultinomialBound},
$$
\Pi(m,(p_{0},\ldots,p_{k}))\leq m^{-\Omega(\sqrt{m})}.
$$
\end{proof}

\section{The Main Theorem}
We recall that $R(G)$ stands for $G$'s $2$-core. We also denote $\tilde R := V\setminus R(G)$.

\begin{lemma} \label{lem:Humongous2Core}
Let $G$ be a $G(n, p)$ graph where $p>\omega(\frac{1}{n})$. For every $\frac{n}{10}>x>\frac{n}{e^{np}}$ there holds
$$\Pr(|\tilde R|\geq x) < e^{-\Omega(npx)}.$$
\end{lemma}

\begin{proof}
Let $S\subseteq V$ be the set of those vertices in $G$ with degree at most $3$. We claim that $|S| \geq \frac{|\tilde R|}{4}$. Clearly $|E(\tilde R)|< |\tilde R|$, since $\tilde R$ is acyclic. Also, $d(\tilde R,R) \leq |\tilde R|$ since a vertex in $\tilde R$ can have at most one neighbor in $R$. Hence,
$$4|\tilde R|-4|S|\leq 4|\tilde R\setminus S|\leq \sum_{v\in \tilde R} d(v) = 2|E(\tilde R)|+d(\tilde R,R) < 3|\tilde R|,$$
as claimed. Thus, it is enough to bound the probability that $|S|\geq \frac 14x$. We fix a set $A$ of $\frac{x}{4}$ vertices and note that a vertex $v \in A$ has $d(v)< 4$ only if $d(v,V\setminus A) < 4$, which holds with probability $\leq e^{-\Omega(np)}$. Thus, the probability that all vertices in $A$ have degree $\leq 4$ is at most $e^{-\Omega(npx)}$. Therefore, the probability that such a set $A$ exists is at most
$$\binom{n}{\frac 14x}e^{-\Omega(npx)} = e^{-\Omega(npx)}$$
finishing the proof.
\end{proof}

\begin{definition}
Let $G=(V,E)$ be an $n$-vertex graph, and $k\ge 3$ an integer. An order $k$ {\em configuration} of $G$ is a pair of functions $(\phi,\psi):[k]\to V$. If $\phi(i)=\psi(i)$ we say that $i$ is a {\em confluence} of $(\phi,\psi)$.

\begin{itemize}
\item A confluence-free configuration $(\phi,\psi)$ is said to have \emph{type I}  when $k\leq\log n$ and $\left(\phi(1),\ldots,\phi(k),\phi(1)\right)$ and $\left(\psi(1),\ldots,\psi(k),\psi(1)\right)$ are simple cycles (in this order).
\item We say that $(\phi,\psi)$ is a \emph{type II} configuration when $\left(\phi(1),\ldots,\phi(k)\right)$ and $\left(\psi(1),\ldots,\psi(k)\right)$ are each a simple path or a simple cycle. Also, $k\leq\log n$, and $1, k$ are the only confluences.
\end{itemize}
\end{definition}
\begin{lemma}
\label{lem:ConfigurationBound}
Let $G=(V,E)$ be a random $G(n,p)$ graph and let $k\leq \log n$. Pick the functions $\phi,\psi:[k]\to V$ uniformly at random. Consider the events
\begin{itemize}
\item
$C_1$ that $(\phi,\psi)$ is a type I configuration.
\item
$C_2$ that $(\phi,\psi)$ is a type II configuration.
\end{itemize}
Then:
\begin{enumerate}
\item
$\Pr(C_1)\leq p^{k}\cdot(\frac{2}{n}+p)^{k}$
\item
$\Pr(C_2)\leq p^k\cdot(\frac{2}{n}+p)^{k-2}\cdot n^{-2}$
\end{enumerate}
\end{lemma}
\begin{proof}
We only prove the first claim. The same argument applies as well to the second case.\\
Denote $\phi(k+1)=\phi(1)$ and $\psi(k+1)=\psi(1)$.
For $i=0,\ldots,k$, we estimate the probability of the events $A_{i}$ that $\psi(j)\psi(j+1)\in E$ for every $1\leq j\leq i$ and $(\phi(1),\ldots,\phi(k+1))$ is a simple cycle in $G$. Clearly, $\Pr(A_0)=\Pr\left(\phi(1),\ldots,\phi(k+1)\text{ is a simple cycle}\right) \leq p^{k}$.

We complete the proof by showing that $\Pr(A_{i+1}|A_{i})\leq\frac{2}{n}+p$. Indeed, suppose that $\psi(i)=\phi(j)$ for some $j$. In this case it is possible that $\psi(i),\psi(i+1)$ are neighbors since $\psi(i+1)$ coincides with either $\phi(j-1)$ or with $\phi(j+1)$, but that happens with probability $\le \frac 2n$. Otherwise, they are neighbors with probability $p$. \end{proof}

\begin{lemma}
\label{lem:CompatibleBound}Let $\omega(\frac 1n)\le p=p(n)\le O(n^{-0.5-\epsilon})$ for some $\epsilon > 0$. Pick a random $G(n,p)$ graph $G=(V,E)$ and two random maps $\phi,\psi:[k]\to V$ where $k\leq \log n $. Let $s$ denote the number of indices $i\in\{1,\ldots,k\}$ such that $\nabla(\phi(i))=\nabla(\psi(i))$.
Then:
\[
\Pr\left(s > \frac 14k\mid C_1\right)\leq(pn)^{-\Omega(\sqrt{pn}\cdot k)}
\]
\[
\Pr\left(s > \frac 14(k-2)\mid C_2\right)\leq(pn)^{-\Omega(\sqrt{pn}\cdot k)}
\]
\end{lemma}
\begin{proof}
We only prove the type I case. The same argument applies to type II configurations as well. The argument below and all relevant calculations take place in the space {\em conditioned on} $C_1$.

Let $T=\text{Image}(\phi)\cup\text{Image}(\psi)$ and $t=|T|$. For each index $i$ let $U_i$ be the set of those neighbors of $\phi(i)$ that have no other neighbor in $\tilde N(T)$. We expose the subgraph induced on $\tilde N(T)$, thus revealing the sets $U_i$. The following proposition comes in handy:
\begin{proposition}\label{prop:CompatibleBoundProp}
With probability $1-e^{-\Omega(npt)}$ there holds:
\begin{itemize}
\item
$|\tilde{N}(T)| < 2npt$
\item
There are at least $\frac{7k}{8}$ indices $k\ge i\ge 1$ for which $\frac{np}{4}\le |U_i|\le 4np$.
\end{itemize}
\end{proposition}
We proceed under the conditioning that the conclusion of this Proposition holds. We next reveal the edges connecting $\tilde N(T)\setminus \bigcup_i U_i$ and $V\setminus\tilde N(T)$. This determines $\nabla(\psi(j))$ for all $j$. On the other hand, $\nabla(\phi(i))$ is completely determined by the neighbor sets of vertices from $U_i$ in $V\setminus\tilde N(T)$. Consequently the family of multisets $\{\nabla(\phi(i))\}_i$ is independent.

We are concerned with the event that $\nabla(\phi(i))=\nabla(\psi(i))$. At this stage this may already be impossible, and if possible, this uniquely determines the multiset of degrees $d(x, V\setminus\tilde N(T))$ over $x\in U_i$. The elements of this multiset are drawn from a binomial distribution, so by Lemma \ref{Lem:MultinomialWithBinomialProbsBound}, if $\frac{np}{4}\le |U_i|\le 4np$, then
$$\Pr(\nabla(\phi(i))=\nabla(\psi(i)))\le (np)^{-\Omega(\sqrt{np})}.$$

Note that for $s>\frac 14k$ to hold, the equality $\nabla(\phi(i))=\nabla(\psi(i))$ must hold for at least $\frac k8$ of the indices $i$ for which $|U_i|\ge \frac{np}{4}$. Hence,
$$\Pr(s>\frac 14k)\leq \binom {\frac 78k}{\frac 18k} (np)^{-\Omega(k\sqrt{np})} \leq (np)^{-\Omega(k\sqrt{np})},$$
as stated.

{\em Proof of Proposition~\ref{prop:CompatibleBoundProp}:~}
The first claim follows from Chernoff's bound, as we observe that
$$|\tilde N(T)| \sim t+\text{Bin}(n-t,1-q^t),\mbox{ where }q=1-p.$$
so that
$$|\mathbb E(|\tilde N(T)|)| \leq npt(1+o(1)).$$
For the second claim
$$|\sigma(T)| \sim \text{Bin}(n-t,tpq^{t-1})$$
and so
$$\mathbb E(|\sigma(T)|) \geq npt(1-o(1)).$$
Let $A$ denote the event that $|\tilde N(T)|\leq 2npt$ and $|\sigma(T)| \ge \frac{npt}{2}$. By Chernoff's bound,
$$\Pr (A) \ge  1-e^{-\Omega(npt)}.$$
Now,
$$d(\sigma(T), N(T)) \sim \text{Bin}\left(|\sigma(T)|\cdot \left(|\tilde N(T)|-t-\frac{|\sigma(T)|+1}{2}\right), p\right)$$

Let $B$ denote the event that $A$ holds, and, in addition, $d(\sigma(T), N(T))<\frac{npt}{8}$. Note that $A$ implies 
$$|\sigma(T)|\cdot \left(|\tilde N(T)|-t-\frac{|\sigma(T)|+1}{2}\right) \leq |\tilde N(T)|^2\leq (2npt)^2.$$
Hence,
$$\Pr(B|A) \geq  1-\binom{(2npt)^2}{\frac {npt}{8}} p^{\frac {npt}{8}}\geq 1- e^{-\Omega(npt)}$$
and so
$$\Pr(B) = \Pr(B|A)\cdot \Pr(A) \geq 1-e^{-\Omega(npt)}.$$
Let $U=\bigcup_i U_i$. Note that $B$ implies that at least $\frac{npt}{4}$ vertices in $\sigma(T)$ have no neighbor in $N(T)$, and thus, $|U|\geq \frac{npt}{4}$.
Clearly, $|U_i| \sim \text{Bin}(|U|, \frac{1}{t})$. Let $D_i$ denote the event that $|U_i| < \frac{np}{4}$. For $x\geq \frac{npt}{4}$, Chernoff's bound implies
$$\Pr(D_i\mid|U|=x)\leq e^{-\Omega(np)}.$$
Note that given $|U|=x$, the event $D_i$ is negatively correlated with every event of the form $\bigcap_{j\in J} D_j$ where $\emptyset\neq J\subseteq [k]\setminus{i}$. Thus, for every $I\subseteq [k]$,
$$\Pr\left(\bigcap_{i\in I} D_i\mid|U|=x\right)\leq \prod_{i\in I} \Pr(D_i\mid|U|=x) \leq e^{-\Omega(np|I|)}.$$
In particular, the event $\tilde D$ that at most $\frac{k}{16}$ of the $D_i$ hold satisfies
$$\Pr(\tilde D\mid|U|=x) \geq 1 - \binom{k}{k/16}\cdot e^{-\Omega(npt)} \geq 1- e^{-\Omega(npt)}$$
which implies
$$\Pr(B\cap \tilde D) \geq 1 - e^{-\Omega(npt)}.$$
Let $F$ be the event that at most $\frac{k}{16}$ of the $U_i$'s satisfy $|U_i|<4npt$. 
A similar argument shows that given
$$\Pr(B\cap \tilde F) \geq 1-e^{-\Omega(npt)}$$
and we conclude that
$$\Pr(B\cap \tilde D\cap \tilde F) \geq 1-e^{-\Omega(npt)}.$$
\end{proof}

At this stage we have already established the following whp: For $\omega(\frac 1n)\le p(n)\le O(n^{-0.5-\epsilon})$ every automorphism of a $G(n,p)$ graph pointwise fixes its $2$-core. However, we seek to prove the stronger statement that the $2$-core has no nontrivial symmetries. As before consider two random maps $\phi,\psi:[k]\to V$ where $3\le k\leq \log n $. Let $T=\text{Image}(\phi)\cup\text{Image}(\psi)$ and define the events $C_1, C_2$ as above. Clearly $T\subseteq R$, since $T$ is a union of cycles, and  now we need to control the effect of non-$2$-core vertices on $\text{aut}(G)$. This effect is mediated by the set $P\subseteq T$ of $T$'s {\em peripheral} vertices, namely those within distance $2$ of $\tilde R = V\setminus R$. As we show, the above-mentioned effect is not large, since $|P|$ tends to be small. We prove

\begin{lemma}
	\label{lem:PeripherialBound}
	\[
	\Pr\left(|P| > \frac k8\mid C_1\right)\leq (np)^{-\omega_n(k)}
	\]
	\[
	\Pr\left(|P| > \frac k8\mid C_2\right)\leq (np)^{-\omega_n(k)}
	\]
\end{lemma}
\begin{proof}
We only prove case I. The same argument applies as well to case II. All our arguments below are made {\em conditioned on} $C_1$.
Let $q=1-p$ and $t=|T|$. Clearly, $k\leq t\leq 2k$.

Reveal the subgraph $H$ of $G$, induced by $V\setminus T$. Denote $W = V(H)\setminus R(H)$. Let $x=\frac{n}{(np)^{\log(np)}}$. By Lemma ~\ref{lem:Humongous2Core},
$$\Pr(|W|\leq x) \ge 1-e^{-\Omega(npx)}\geq 1-(np)^{-\omega(k)}.$$ We henceforth condition on this event. Note that $\tilde R\subseteq W$, and thus, it is enough to bound the number of vertices in $T$ at distance $\le 2$ from $W$. We denote $N_H(W)$ by $Q$. We claim that $|Q|\leq |W|$, since every vertex in $Q$ has a neighbor in $W$, whereas every vertex in $W$ has at most one neighbor in $Q$. (Note that $Q\subseteq R(H)$ and a vertex with more than one neighbor in $Q$ is in $R(H)$ as well). 

To understand the set $P$ of peripheral vertices, we define three sets $P_1, P_2, P_3$ with $P\subseteq P_1\cup P_2\cup P_3$ and show that whp all $|P_i|$ are small.
Let $P_1$ be the set of those vertices in $T$ with a neighbor in $W$. Let $P_2$ be the set of those vertices in $T$ with a neighbor in $P_1$. Finally, $P_3$ is the set of those vertices in $T$ with a neighbor in $Q$.

Now reveal the set of cross edges $E(T,V\setminus T)$. For $v\in T$, the probability that $v$ has a neighbor in $W$ is at most $xp$. Thus, 
$$\Pr\left(|P_1|\geq \frac{k}{400}\right)\leq \binom{t}{\frac{k}{400}}(xp)^{\frac{k}{400}}\leq (np)^{-\Omega(k\log(np))}$$
and similarly, $\Pr\left(|P_3|\geq \frac{k}{400}\right)\leq(np)^{-\Omega(k\log(np))}$. In what follows we condition on the event that $|P_1|,|P_3|\leq \frac{k}{400}$.

We finish by bounding $|P_2|$. Reveal the edge set $E(P_1,T)$. By assumption, $T$ is the image of a type I configuration, namely two simple cycles, possibly with some overlaps. This implies the existence of certain edges in $E(P_1,T)$, at most $4|P_1|$ in number. In addition, the random variable $d(P_1,T)$ is a sum of at most $|P_1||T|$ independent Bernoulli-$p$ random variables. By assumption $4|P_1|\leq \frac{k}{100}$, so that $d(P_1,T) > \frac{k}{50}$ only if at least $\frac{k}{100}$ of these Bernoulli trials succeed. Therefore
$$\Pr\left(d(P_1,T) > \frac{k}{50}\right) \leq \binom {\frac{kt}{400}}{\frac k{100}} \cdot p^{\frac k {100}}\leq  (kp)^{\Omega(k)}\leq (np)^{-\omega_n(k)}$$
Clearly, $|P_2| \leq d(P_1,T)$, and so, $|P_2|\leq \frac k {50}$ with probability at least $1-(np)^{-\omega_n(k)}$.
\end{proof}

\begin{definition}
	A configuration $(\phi,\psi)$ of $G$ is said to be \emph{compatible} if there
	exists an automorphism $\pi$ of $R(G)$ such that $\pi(\phi(i))=\psi(i)$
	for each $i$.
\end{definition}

\begin{lemma}
\label{lem:NoBadConfigs}Let $G$ be a random $G(n,p)$ graph with $\omega(\frac 1n)\le p\le O(\frac {\log n} n)$. Then whp $G$ contains no compatible configuration of type I or II.
\end{lemma}
\begin{proof}
We prove the claim for type I configurations. The proof for type II follows the same argument.

In the coming paragraph we denote $\nabla(v)$ by $Y(v)$. We also consider the 2-neighborhood of $v\in R(G)$ in the subgraph induced by $R(G)$ and denote $\nabla(v)$ in that graph by $Z(v)$. Clearly, a configuration $\phi,\psi:[k]\to V$ can be compatible only if $Z(\phi(i))=Z(\psi(i))$ for each $1\leq i\leq k$. 

Let $3\leq k\leq \log n$ and pick two functions $\phi,\psi:[k]\to V$ uniformly at random. By Lemma \ref{lem:ConfigurationBound}, the probability that $(\phi,\psi)$ is a configuration is at most $p^{k+o_n(1)}$. Conditioned on this event, let $$A = \{i\in[k]\mid Y(\phi(i))\neq Y(\psi(i))\}.$$ By Lemma \ref{lem:CompatibleBound}, $\Pr(|A|< \frac 34k) \le (np)^{-\omega_n(k)}$. Let $$B=\{i\in[k]\mid Y(\phi(i))=Z(\phi(i))\text{~and~} Y(\psi(i))=Z(\psi(i))\}.$$ Note that $i\in B$ when both $\phi(i)$ and $\psi(i)$ are non-peripherial. Hence, by Lemma~\ref{lem:PeripherialBound}, $\Pr(|B|< \frac 34k) \le (np)^{-\omega_n(k)}$. But $|A|,|B|\geq \frac 34 k$, so they must intersect, say $i\in A\cap B$. Then $Z(\phi(i))\neq Z(\psi(i))$, which makes $(\phi,\psi)$ incompatible. Clearly this holds with probability $1-(np)^{-\omega_n(k)}$.

If $a_k$ is the number of compatible type I configurations we can now estimate its expectation:
$$ \mathbb E(a_k) = n^{2k} \cdot p^{2k+o_n(1)}\cdot (np)^{-\omega_n(k)}\leq (np)^{-\omega_n(k)}$$
and so
$$
\sum_{k=3}^{\log n} \mathbb E(a_k) \leq (np)^{-\omega_n(1)}
$$
which completes the proof.
\end{proof}

We can now finish up the proof of our main theorem.

\mainTheorem*
\begin{proof}
Let $H$ denote the $2$-core of $G$.
It is known (\cite{CL01}) that whp $\text{diam}(G)<\frac{\log n}{2}$, which we henceforth assume.

Suppose that $\pi(v)\neq v$ for some $\pi\in\text{aut}(H)$, and a vertex $v\in R(G)$. By Lemma \ref{lem:NoBadConfigs}, it is enough to show that this assumption implies that $G$ has a compatible configuration. It is easy to see that if $\pi$ fixes all vertices of $H$ contained in cycles, then $\pi$ is trivial, so let $C$ be a cycle that contains $v$. The bound on $G$'s diameter implies that such a $C$ exists of length at most $\log n$. 

The argument splits now according to whether $\pi$ fixes some vertex in $C$. If it does not, then $\phi$ and $\psi$ that map $[k]$ to $C$ and to $\pi(C)$ respectively form a compatible type I configuration, and we are done. Otherwise, consider an arc $\Gamma=u\rightsquigarrow u'$ (possibly $u'=u$) of $C$ so that: $v\in\Gamma$, and the only $\pi$-fixed points in $\Gamma$ are $u, u'$. We obtain a compatible type II configuration by letting $\phi$ map $[k]$ to $\Gamma$ and $\psi$ map $[k]$ to $\pi(\Gamma)$.

\end{proof}

\section{Connections with the Reconstruction Problem}

The purpose of this section is to prove:
\begin{theorem}\label{thm:reconstruct}
For every $0\leq p\leq 1$ whp a $G(n,p)$ graph is reconstructible.
\end{theorem}

We may clearly restrict ourselves to the range $0\leq p\leq \frac 12$, since a graph is reconstructible iff its complement is reconstructible. We may further restrict our attention to the range $\frac {(1-\epsilon)\log n}{n}\leq p\leq \frac {(5/2+\epsilon)\log n}{n}$ since the theorem is known for the two complementary ranges. For $p\geq \frac {(5/2+\epsilon)\log n}{n}$ this was done by  Bollob\'as \cite{Bol90}. Also, disconnected graphs are reconstructible \cite{Bon91}, which takes care of the range $p\leq \frac {(1-\epsilon)\log n}{n}$. One further simplification is that for $p$ in the above range, $G$ almost surely has no $K_{3,2}$ subgraph. So we can and will be assuming this below. Our line of argument resembles the first part of the proof of  Theorem~\ref{thm:mainTheorem}. However, we need to adapt Lemma~\ref{lem:CompatibleBound}, a key step in that proof. This lemma gives an upper bound on $\Pr(\nabla(\phi(i))=\nabla(\psi(i)))$, while here this equality gets replaced by an {\em approximate} equality as we now define.

For two multisets of integers we say that $A\approx B$ if they can be made equal by applying some of the following operations to each of them. (Here $X$ refers to either $A$ or $B$).
\begin{itemize}
	\item Decrease some elements of $X$ by $1$ or $2$. The total subtracted sum must be $\le 4$.
	\item Delete one or two elements of $X$.
\end{itemize}

\begin{definition}
A configuration $(\phi,\psi)$ is {\em acceptable} if there exist vertex sets $U,W\subseteq V$ of size $n-2$ such that $\text{im}(\phi)\subseteq U$, $\text{im}(\psi)\subseteq W$, and $G_U$ and $G_W$ are isomorphic through a graph isomorphism $\pi$ that maps $\phi(i)$ to $\psi(i)$ for every $i$. 
\end{definition}

\begin{lemma} \label{lem:NoAcceptableConfigs}
Whp, $G$ contains no acceptable configurations of type I or II. 
\end{lemma}
\begin{proof}
We first claim that $\nabla_G(u)\approx \nabla_G(\pi(u))$, for every $u\in U$ for $U$, $W$ and $\pi$ as above. This is so, since the property of $\pi$ implies $\nabla_{G_U}(u)=\nabla_{G_W}(\pi(u))$. These are subgraphs of $n-2$ vertices and the effect of the two missing vertices is limited due to $K_{2,3}$-freeness. Since $G$ is $K_{3,2}$ free, $|N(u)\cap N(v)|\leq 2$ for every $v\in V\setminus U$. Hence, by removing $v$ from $G$ the possible changes in $\nabla(u)$ are: (i) Decreasing one or two elements of $\nabla(u)$ by $1$: Each vertex in $N(u)\cap N(v)$ (of which there are at most two) may lose one neighbor, (ii) Removal of a single element from $\nabla(u)$ (the element corresponding to $v$ itself, if $uv\in E$). 

To prove the Lemma, we first strengthen Lemma \ref{lem:CompatibleBound}, and replace the condition $\nabla(\phi(i))=\nabla(\psi(i))$ by $\nabla(\phi(i))\approx \nabla(\psi(i))$. The proof is essentially the same, with one change: Clearly the multiset $\{d(x,V\setminus \tilde N(T))\mid x\in U_i\}$ is uniquely determined by the condition $\nabla(\phi(i))=\nabla(\psi(i))$. Now we operate under the weaker condition $\nabla(\phi(i))\approx \nabla(\psi(i))$. Rather than the above multiset, we consider a multiset where at most two of the entries are $"\ast"$ which stand for the possibly deleted vertices. This multiset can take on only $\text{poly}(np)$ possible values. Lemma \ref{Lem:MultinomialWithBinomialProbsBound} and a union bound argument yield: 
$$\Pr(\nabla(\phi(i)) \approx\nabla(\psi(i)))\leq (np)^{-\Omega(\sqrt{np})}\cdot(np)^{O(1)} = (np)^{-\Omega(\sqrt{np})}.$$

By Lemma \ref{lem:ConfigurationBound} and the stronger version of Lemma \ref{lem:CompatibleBound} proved here, the expected number of acceptable type I or type II configurations in $G$ is at most
$$\sum_{k=3}{\log n} (np)^{-\omega_n(k)} \leq (np)^{-\omega_n(1)}.$$
\end{proof}

\begin{definition}
We say that a vertex pair $u,v\in R(G)$ is {\em interior} if $R(G\setminus\{u,v\}) = R(G)\setminus\{u,v\}$.
\end{definition}

\begin{lemma} \label{lem:NoPartialsIsomorphisms}
Whp, for every interior vertex pair $\{u,v\}$ it holds that (i) every automorphism of $G_{V\setminus\{x,y\}}$ fixes $R(G)\setminus\{u,v\}$ and (ii) For every interior vertex pair $\{x,y\}\neq \{u,v\}$, the graphs $G_{V\setminus\{u,v\}}$ and $G_{V\setminus\{x,y\}}$ are non-isomorphic.
\end{lemma}
\begin{proof}
We may assume that $\text{diam}(G) < \frac {\log n}8$, as this holds whp~\cite{CL01}. Also,\\ \mbox{$\text{diam}(G_{V\setminus\{u,v\}})< \frac {\log n}2$} (likewise for $\{x,y\}$) since the removal of a vertex at most doubles the diameter. By Lemma \ref{lem:NoAcceptableConfigs}, we may also assume that $G$ has no acceptable type I or II configurations. 

We prove both parts of the Lemma together by considering as well the case $\{x,y\}=\{u,v\}$. Assume that there exists an isomorphism $\pi$ between $G_{V\setminus\{u,v\}}$ and $G_{V\setminus\{x,y\}}$ that does not pointwise fix the $2$-core. To prove the Lemma, it is enough to show that there exists an acceptable type I or II configuration in $G$. We consider two cases, first  where $\pi$ moves some vertex in $R(G_{V\setminus\{u,v\}})$ that resides in a cycle. Since $\text{diam}(G_{V\setminus\{u,v\}})< \frac {\log n}2$, we may assume that this cycle has length $<\log n$. The existence of an acceptable type I or II configuration follows from the same argument as that in Theorem \ref{thm:mainTheorem}. In the second case, $\pi$ fixes pointwise every cycle of $R(G_{V\setminus\{u,v\}})$. Thus, it must map some path between two vertices in cycles, fixed by $\pi$, to a different path between these two vertices. Due to the bound on the diameter of $G_{V\setminus\{u,v\}}$, the length of these paths must be $<\frac {\log n}2$, which yields a type II acceptable configuration.  
\end{proof}

\begin{proof}[Proof of Theorem \ref{thm:reconstruct}]

We may and will assume that $G$ satisfies the conclusion of Lemma \ref{lem:NoPartialsIsomorphisms}. For $u\in U$, we denote $\tilde G_u := G_{V\setminus\{u\}}$.

We claim that the cardinality $|R(G)|$ is reconstructible. Indeed, it is known~\cite{Bon91} that the degree sequence of $G$ is reconstructible, and thus, the property $R(G)=V$ is recognizable. Now, assume that $R(G)\neq V$. It is clearly possible to determine $d(u)$ from $\tilde G_u$. Also $R(G)=R(\tilde G_u)$ when $d(u) = 1$. Since $u\in R(G)$ iff $|R(\tilde G_u)|< |R(G)|$, we can determine whether $u\in R(G)$ by observing $\tilde G_u$. 

We also note that the degree sequence of $G$'s $2$-core is reconstructible. Indeed, if $V=R(G)$ this follows from the reconstructibility of $G$'s degree sequence. Otherwise, the $2$-core itself is reconstructible, as above. 

Let  $A = \{u\in R(G)\mid d(v,R(G))\geq 4 \text{~for all~} v\in N(u)\}$. Note that every vertex pair in $A$ is interior. It is not hard to determine whether $u\in A$ given $\tilde G_u$, based on the reconstructibility of the $2$-core's degree sequence. We claim that $A$ contains almost all vertices. By Lemma \ref{lem:Humongous2Core}, there holds whp $|R(G)|\geq n-o(n)$. For $v\in V$
$$\Pr(d(v,R(G))\leq 3) \leq O(np^3e^{-np})$$
and by the union bound
$$\Pr(\exists v~~v\in N(u)\wedge d(v,R(G))\leq 3) \leq O(np^4e^{-np}) \leq o(1).$$
So, let $v'\in V\setminus\{u\}$ and $u'\in V\setminus\{v\}$ be such that $\{u,v'\}$ and $\{v,u'\}$ are interior pairs, and there exists an isomorphism $\pi$ between $G_{V\setminus\{u,v'\}}$ and $G_{V\setminus\{u',v\}}$. By Lemma \ref{lem:NoPartialsIsomorphisms}, this holds only when $u=u'$, $v=v'$ and $\pi$ fixes the $2$-core pointwise. Using this property, we can identify the vertices $v$ and $u$ respectively in the graphs $\tilde G_u$ and $\tilde G_v$ and identify each vertex in the $2$-core of one graph with its counterpart in the other. This allows us to reconstruct $G$ up to the question of whether $uv\in E$. Since $|E|$ is reconstructible, this last question can be answered as well.
\end{proof}

\section{Connections with the Canonical Graph Labeling Problem} 
In this section we describe a polynomial time random graph canonical labeling algorithm for graphs in $G(n,p)$ where $\omega(\frac 1n)\leq p\leq n^{-(0.5+\epsilon)}$.

Let ${\cal C}$ be the collection of all {\em rooted, oriented} cycles of length $3\leq k\leq \log n$ in an $n$-vertex graph $G=(V,E)$. We use $A \leq B$ to denote the lexicographic ordering between multisets of integers, where the elements in $A$ and in $B$ appear in increasing order. We equip $\cal C$ with the semi-order $\prec$ where short cycles precede longer ones. For two cycles $X=(x_1,\ldots, x_k), Y=(y_1,\ldots, y_k)\in \cal C$ we say that $X \prec Y$ if for some $i$ there holds $\nabla(x_i)<\nabla(y_i)$ and $\nabla(x_j)=\nabla(y_j)$ for every $1\leq j<i$.

We claim that for the relevant range of $p$, a $G(n,p)$ graph satisfies the following conditions whp:

\begin{enumerate}
\item Each connected component of $G_{V\setminus R(G)}$ is a tree of size $\leq \log n$.
\item $\text{diam}(G)<\frac{\log n}{2}$.
\item ${\cal C}(G)$ is totally ordered by $\prec$.
\end{enumerate}
Property (1) is easy to derive by a first-moment argument. For property (2), see \cite{CL01}. A proof of Property (3) follows from property (2) by a simple variation of the proof of Theorem \ref{thm:mainTheorem}.

We now explain how to canonically label a graph $G=(V,E)$ with these three properties. To a vertex $v$ that is contained in a cycle we assign the $\prec$-smallest label $X=(x_1,\ldots, x_k)\in \cal C$ over all cycles for which $v=x_1$. This label can be found in polynomial time. If $k$ is the length of the shortest cycle through $v=x_1$, then it is easy to show that there are at most $n^3$ such cycles. We scan all of them and pick the $\prec$-smallest one.

Note next, that a vertex $v\in R(G)$ that is not contained in a cycle must reside on the unique path between two vertices $u,w\in V$, each contained in a cycle. Therefore $v$ is uniquely defined by its distances from $u$ and from $w$. This, and the labels of $u$ and $w$, give us a unique label for $v$.

Finally we find labels for vertices in $V\setminus(R(G))$. By property (1), such a vertex belongs either to (i) a tree of size $\leq \log n$ rooted at some vertex of $R(G)$ or (ii) an acyclic connected component of size $\leq \log n$. Let $v$ be a vertex of type (i), belonging to a tree $T$ rooted at $u\in R(G)$. There are only $\text{poly}(n)$ rooted trees of size $\leq \log n$ \cite{Ott48}, so we can list them and give a unique polynomial-length label to each vertex of each such class. We label $v$ by a pair $(x,y)$, where $x$ is the label of the vertex corresponding to $v$ in $T$'s isomorphism class in the list, and $y$ is $u$'s label. That is, $v$ is labeled as "The vertex of type $x$ in the tree rooted at $u$". Type (ii) vertices are likewise handled, using a list of all isomorphism classes of non-rooted trees. To deal with vertices on acyclic connected components, collect all connected components of the same isomorphism class and give each of them a unique number. The label of $v$ consists of the type of tree that contains it, that tree's ordinal number in its isomorphism class, and $v$'s location in that tree.

\section{Discussion and Open Problems}
For smaller values of $p$ the structure of $\text{aut}(G)$ may become somewhat more complicated. For $p=\Theta(\frac 1n)$, a $G(n,p)$ graph has, with probability bounded away from zero and one, some small symmetric components, e.g., an isolated triangle. Moreover, with probability $\in(0,1)$ even the $2$-core of the graph's giant component, has a nontrivial symmetry. This may result e.g., from a triangle that "hangs off" the $2$-core. However, as shown in \cite{pit90}, whp this $2$-core has a unique biconnected component of $\Omega(n)$ vertices. We suspect that this giant biconnected component is rigid whp.

For $\frac {(5/2+\epsilon) \log n}n \leq p\leq \frac 12$ it was shown by Bollob{\'a}s \cite{Bol90} that not only is $G$ reconstructible whp, such graphs have reconstruction number three. We do not know whether this holds as well for smaller and substantially smaller values of $p$.

\bibliography{}

\end{document}